\documentclass[12pt,leqno]{article}
\usepackage{amsmath,amsfonts,amsthm}
\usepackage{amssymb}

\newtheorem{thm}{Theorem}[section]
\newtheorem{cor}[thm]{Corollary}
\newtheorem{lem}[thm]{Lemma}
\newtheorem{prop}[thm]{Proposition}

\theoremstyle{definition}

\newtheorem{rem}[thm]{Remark}

\numberwithin{equation}{section}

\newcommand{\N}{\mathbb{N}} 
\newcommand{\Z}{\mathbb{Z}} 
\newcommand{\R}{\mathbb{R}} 

\newcommand{\abs}[1]{\lvert#1\rvert}

\newcommand{\norm}[1]{\lVert#1\rVert}

\newcommand{\ldens}{\operatorname{\underline{dens}}}

\def\llbracket{[\![}
\def\rrbracket{]\!]}

\multlinegap=\parindent
\multlinetaggap=\parindent

\def\mfrac#1#2{{#1/#2}}

\def\comment#1{}
\def\hods{\ \,}
\def\squ{\square}
\def\loo#1 #2\par{\looseness#1{\spaceskip.33em plus.22em minus.17em#2\par}\par}

\newcommand{\wc}[1]{\setbox0\hbox{\indent#1\enspace}\leftmargini\wd0}

\newenvironment{aitemize}[1]{\begin{list}{}%
{\def\makelabel##1{{\rlap{##1}}\hss}%
\setlength{\topsep}{4pt}%
\setlength{\parsep}{0pt}%
\setlength{\itemsep}{0pt}%
\settowidth{\labelwidth}{#1}%
\setlength{\leftmargin}{15pt}%
\addtolength{\leftmargin}{\labelwidth}%
\addtolength{\leftmargin}{.5em}%
}}
{\end{list}}

\begin{document}

\comment{
\inicaut1{E. M.}
\imaut1{Elisabetta M.}
\nazaut1{Mangino}
\adraut1{Dipartimento di Matematica e Fisica ``Ennio De Giorgi''\cr
Universit\`a del Salento\cr
 I-73100 Lecce, Italy}
\emailaut1{elisabetta.mangino@unisalento.it}
\inicaut2{M.}
\imaut2{Marina}
\nazaut2{Murillo Arcila}
\adraut2{IUMPA\cr
Universitat Polit\`{e}cnica de Val\`{e}ncia\cr
Edifici 8G, Cam\'{\i} Vera S/N\cr
E-46022 Val\`{e}ncia, Spain}
\emailaut2{mamuar1@posgrado.upv.es}
\abbrevtitle{Frequently hypercyclic translation semigroups}
\dedykacja{Dedicated to Jos\'e Bonet on the occasion of his 60th birthday}
\title{Frequently hypercyclic translation semigroups}
\tyt{\im1\ \naz1\ {\rm (Lecce)\ and}\cr \im2\ \naz2}{Val\`{e}ncia}
\ab
}


\baselineskip=17pt


\title{Frequently hypercyclic translation semigroups}

\author{Elisabetta M. Mangino\\ 
Dipartimento di Matematica e Fisica  "Ennio De Giorgi"\\
Universit\`a del Salento\\
 I-73100 Lecce, Italy\\
E-mail: elisabetta.mangino@unisalento.it
\and 
Marina Murillo Arcila\\
IUMPA, Universitat Polit\`{e}cnica de Val\`{e}ncia \\
Edifici 8G, Cam\'{\i} Vera S/N\\
E-46022 Val\`{e}ncia, Spain\\
E-mail: mamuar1@posgrado.upv.es}

\date{{\emph{Dedicated to Jos\'e Bonet on the occasion of his 60th birthday}}}

\maketitle


\renewcommand{\thefootnote}{}



\renewcommand{\thefootnote}{\arabic{footnote}}
\setcounter{footnote}{0}


\begin{abstract}
Frequent hypercyclicity for translation $C_0$-semigroups on weighted
spaces of continuous functions is studied. The results are achieved by
establishing an analogy between frequent hypercyclicity for
translation semigroups and for weighted pseudo-shifts and by
characterizing frequently hypercyclic weighted pseudo-shifts on spaces
of vanishing sequences. Frequently hypercyclic translation semigroups
on weighted $L^p$-spaces are also characterized.
\end{abstract}

\section{Introduction and preliminaries}

A continuous linear operator $T$ on a separable Banach space $X$ is
called \emph{hypercyclic} if there is an element $x\in X$, called
a \emph{hypercyclic vector}, such that the orbit $\{T^n x \, :\, n\in\N\}$ is
dense in $X$. The first historically known examples of hypercyclic
operators are due to Birkhoff, MacLane and Rolewicz. In particular,
the last author studied hypercyclicity  of weighted
shift operators on $l^p$ and $c_0$. The interest in the study of
linear dynamics of shift operators is nowadays still alive, since many
classical operators (e.g. derivative operators in spaces of entire
functions) can be viewed as such operators. We refer to the recent
monographs \cite{bayart_matheron2009dynamics} and
\cite{grosse-erdmann_peris2011linear} for a complete overview on the
subject.

In 2005, motivated by Birkhoff's ergodic theorem, Bayart and Grivaux \cite{BaGr06}
introduced the notion of frequently hypercyclic
operators, trying to quantify how ``often'' an orbit meets non-empty
open sets. More precisely, if the \emph{lower density} of a set $A\subseteq
\N$ is defined as
\[
\ldens (A):=\liminf_{N\rightarrow \infty} \#\{n\leq N:n\in A\}/N,
\]
an operator $T\in L(X)$ is said to be \emph{frequently
hypercyclic} if there exists $x\in X$ (called a \emph{frequently hypercyclic vector}) such that,
for every non-empty open subset $U\subseteq X$,\vspace{-4pt}
$$
\ldens(\{n\in\mathbb{N}:T^nx\in U\})>0.
$$ 
This notion has been deeply investigated by various authors: see
e.g. \cite{Gri06, bonilla_grosse-erdmann2007frequently, DFGP12}. In
particular frequently hypercyclic weighted shifts have been
investigated in \cite{BayartGrivaux07invariant,
  bonilla_grosse-erdmann2007frequently}; their behaviour in $l^p$ and $c_0$ has
been completely characterized by Bayart and Ruzsa~\cite{bayartruzsa}.

In parallel with the theory for linear operators, since the seminal
paper by Desch, Schappacher and Webb \cite{deschschappacher97}, many
researchers turned their attention to the hypercyclic behaviour of
strongly continuous semigroups. Actually hypercyclicity appears in
solution semigroups of evolution problems associated with ``birth and
death'' equations for cell populations, transport equations, first
order partial differential equations, Black--Scholes equation, and
diffusion operators like Ornstein--Uhlenbeck operators
\cite{aroza2014,banasiak_lachowicz_moszynski2006,banasiak_lachowicz_moszynski2007,banasiak_moszynski2011,
  bayart_bermudez, bermudez_et_alt, conejero_mangino,
  conejero_peris_dcds_09, emamirad, kalmes, matsuitakeo}.

We recall that, if $X$ is a separable infinite-dimensional Banach
space, a $C_0$-semigroup $(T_t)_{t\geq 0}$ of continuous linear
operators on $X$ is said to be \emph{hypercyclic} if there exists
$x\in X$ (called a \emph{hypercyclic vector} for the semigroup) such that the
set $\{T_tx\,:\,t\geq 0\}$ is dense in~$X$.  An element $x\in X$ is
said to be a \emph{periodic point} for the semigroup if there exists
$t>0$ such that $T_tx=x$.  A semigroup $(T_t)_{t\geq 0}$ is called
\emph{chaotic} if it is hypercyclic and the set of periodic points is
dense in~$X$.

The role of a ``test'' class, which is played by weighted shifts in the
setting of discrete linear dynamical systems, is taken over by
translation semigroups in the setting of continuous linear dynamical
systems.

Let $I=\R$ or $I=[0,\infty\mathclose{[}$.
An \emph{admissible weight function} on $I$ is a 
measurable function $\rho : I\rightarrow \mathopen{]}0,\infty\mathclose{[}$ for which
there exist constants $M \geq 1$, $\omega \in \R$ such that
$\rho(\tau)\leq Me^{\omega t} \rho(\tau+t)$ for all
$\tau\in I$ and  $t > 0$.

If $\rho$ is an admissible weight function, then for every $l>0$ there exist $A, B>0$
such that for all  $\sigma\in I$ and $t\in [\sigma, \sigma + l]$, 
\begin{equation}\label{admissibility}
A\rho(\sigma) \leq \rho(t)\leq B\rho(\sigma +l).
\end{equation}
For any $1\leq p<\infty$, consider the following function spaces:\vspace{-2pt}
\[
L_p^\rho(I)=\{ u: I\rightarrow \R\mid u \mbox{ is measurable and } \norm{ u}^\rho_p<\infty\},\vspace{-2pt}
\]
where $\norm{ u}^\rho_p=(\int_{I}\abs{u(t)}^p\rho(t) \,dt)^{\mfrac 1 p}$, and\vspace{-2pt}
\[
C_0^\rho(I)=\Bigl\{u:I \rightarrow \R \Bigm{|}
u \mbox{ is continuous and } \lim_{x\to\pm \infty}u(x)\rho(x)=0\Bigr\},\vspace{-2pt}
\]
with $\norm{u}^\rho_\infty = \sup_{t\in I}\abs{u(t)}\rho (t)$.

If $X$ is any of the spaces above and $\rho$ is an admissible weight
function, the \emph{translation semigroup} $\mathcal{T}=(T_t)_{t\geq 0}$ is
defined as usual by\vspace{-2pt}
\[ 
T_tf(x)=f(x+t),\quad\  t\geq 0,\, f\in X,\,  x\in \R,\vspace{-2pt}
\]
and it is a $C_0$-semigroup (see e.g. \cite{deschschappacher97}).

Hypercyclicity and chaos for translation semigroups have been
characterized in \cite{deschschappacher97,matsuitakeo}. In particular,
if $X$ is one of the spaces $L_p^\rho(\R)$ or $C_0^\rho(\R)$ with an
admissible weight function $\rho$, then the translation semigroup
$\mathcal{T}$ on $X$ is hypercyclic if and only if for each
$\theta\in\R$ there exists a sequence $(t_j)_j$ of positive real
numbers tending to $\infty$ such that
\[ 
\lim_{j\to\infty} \rho(t_j + \theta) = \lim_{j\to\infty} \rho(-t_j+\theta) =0.
\]

If $X=C_0^\rho(\R)$, then the translation semigroup $\mathcal{T}$ on
$X$ is chaotic if and only if $\lim_{x\to \pm \infty}\rho(x)=0$.

If $X=L_p^\rho(\R)$, then $\mathcal T$ is chaotic if and only if for all 
$\varepsilon,l > 0$ there exists $P>0$ such that
\[ 
\sum_{k\in\Z\setminus\{ 0\}} \rho (l+kP)<\varepsilon.
\] 

The concept of frequent hypercyclicity was extended to $C_0$-semigroups in~\cite{BaGri07}.

The \emph{lower density} of a measurable set $M\subseteq
\mathbb{R}_+$ is defined by 
\[
\mathop{\underline{{\rm Dens}}} (M):=\liminf_{N\rightarrow \infty} \mu(M\cap [0,N])/N,
\] 
where $\mu$ is the Lebesgue measure on $\mathbb{R}_+$.

A $C_0$-semigroup $(T_t)_{t\geq 0}$ on a separable Banach space $X$ is
said to be \emph{frequently hypercyclic} if there exists $x\in X$
(called a \emph{frequently hypercyclic vector} for the semigroup) such that
$\underline{{\rm Dens}} (\{t\in \mathbb{R}_+:T_tx\in U\})>0$ for any
non-empty open set $U\subseteq X$.  In
\cite{conejero_muller_peris,mangino_peris2011frequently}, it was
proved that $x\in X$ is a (frequently) hypercyclic vector for
$(T_t)_{t\geq 0}$ if and only if $x$ is a (frequently) hypercyclic
vector for each single operator $T_t$, $t>0$. However, this is not the
case in general if we consider the chaos property~\cite{bayart_bermudez}.

In \cite{mangino_peris2011frequently},  a continuous
version of the Frequent Hypercyclicity Criterion was proved, based on the Pettis
integral and the fact that chaotic translation semigroups on weighted spaces of
integrable functions are frequently hypercyclic.

Moreover, in \cite{murillo_peris????strong}, it is proved that the
Frequent Hypercyclicity Criterion for semigroups implies the existence
of strongly-mixing Borel probability measures with full support.

In this paper we characterize, in the spirit of \cite{bayartruzsa},
frequently hypercyclic translation semigroups on $L^p_\rho(I)$ and,
for $\sup_k\rho(k+1)/\rho(k)<\infty$, on
$C_0^\rho(I)$. The main results are Theorems \ref{c0fh} and \ref{lpfh}, 
proved in the last section. In particular, Theorem \ref{c0fh}
will be a consequence of Theorem \ref{pseudoshift} which
characterizes frequent hypercyclicity of the so-called pseudo-shifts
on $c_0(I)$ spaces, where $I$ is a countably infinite set.

\section{Frequently hypercyclic weighted pseudo-shift}

We recall the concept of weighted pseudo-shift  introduced by Grosse-Erd\-mann~\cite{grosse-erdmann2000hypercyclic}.

Given topological sequence spaces $X, Y$ over countably infinite sets
$I$ and $J$ respectively, a continuous linear operator $T:X\rightarrow Y$ is called
a \emph{ weighted pseudo-shift} if there is a sequence $(b_j)_{j\in
  J}$ of non-zero scalars and an injective mapping $\phi:J\rightarrow
I$ such that 
\[
T[(x_i)_{i\in I}]=(b_jx_{\phi(j)})_{j\in J}\quad\ \text{for}\ (x_i)_{i\in I}\in X.
\]

We will be interested in weighted pseudo-shifts acting on spaces of
vanishing sequences. More precisely, given a countable set~$I$, we
consider the space
\[ 
c_0(I)=\{(x_i)_{i\in I}\in \R^I \mid  \forall \varepsilon>0\ \exists J\subseteq I,\, J\ \mbox{finite}\ 
\forall i\in I\setminus J: |x_i|<\varepsilon\},
\]
endowed with the norm $\|(x_i)_{i\in I}\|=\sup_{i\in I}|x_i|.$

Obviously, if  $(W_p)_{p\in\N}$ is an increasing sequence of finite subsets of $I$ such that $I=\bigcup_{p=1}^\infty W_p$, then 
\begin{equation}
c_0(I)=\{(x_i)_{i\in I}\in \R^I \mid  \forall \varepsilon>0\ \exists n\in\N 
\forall i\in I\setminus W_n: |x_i|<\varepsilon\}.
\end{equation}

The first result that we prove is a characterization of frequently
universal sequences of weighted pseudo-shifts on $c_0(I)$.

We recall that a sequence $(T_n)_{n\in\N}$ of continuous mappings
between topological spaces $X$ and $Y$ is said to be \emph{frequently
universal} if there exists $x\in X$, called a \emph{frequently universal
vector} for the sequence, such that for every non-empty open set
$U\subseteq Y$, 
$$
\ldens(\{n\in\N : T_nx\in U\})>0.
$$

Following the idea of Bayart and Ruzsa  \cite{bayartruzsa} for
weighted backward shifts on $c_0(\Z)$, we first obtain
a characterization for weighted pseudo-shifts.

\begin{thm}\label{pseudoshift}
Let $(T_n)_{n\in\N}$ be a sequence of weighted pseudo-shifts on $c_0(I)$ defined by $T_n[(x_i)_{i\in I}]=(b_i^nx_{\phi_n(i)})_{i\in I}$, where the $b_i^n$ are positive real numbers. Assume that:

\wc{\rm(iii)}
\begin{itemize}
\item[\rm(i)] $(\phi_n)_n$ is a run-away sequence, i.e. for any finite subsets $I_0, J_0\subseteq I$ there exists $n_0\in\N$ such that, for every $n\geq n_0$, $\phi_n(J_0)\cap I_0=\emptyset$,
\item[\rm(ii)] there exists $\rho>1$ such that $\mfrac{1}{\rho^{|n-m|}}\leq \mfrac{b_s^n}{b_t^m}$ for all $n, m\in\N$ and $s, t\in I$ such that $\phi_n(s)=\phi_m(t)$,
\item[\rm(iii)] there exists $g:I\rightarrow\R$ such that $|n-m|\leq|g(s)-g(t)|$ for all $n, m\in\N$ and $s, t\in I$ such that $\phi_n(s)=\phi_m(t)$,
\item[\rm(iv)] $(W_p)_{p\in\N}$ is an increasing sequence of finite subsets of $I$ such that $I=\bigcup_{p=1}^\infty W_p$.
\end{itemize}
Then $(T_n)_{n\in\N}$ is frequently universal on $c_0(I)$ if and only if there exist a sequence $(M(p))_{p\in \N}$ of positive real numbers tending to $\infty$ and a sequence $(E_p)_{p\in \N}$ of subsets of $\N$ such that:

\begin{aitemize}{\rm(a)}
\item[\rm(a)] for any $p\geq 1$, $\ldens(E_p)>0$,
\item[\rm(b)] for any distinct $p,q\geq 1$, $n\in E_p$ and $m\in E_q$, $\phi_n(W_p)\cap \phi_m(W_q)=\emptyset$,
\item[\rm(c)] for every $p\geq 1$ and every $s\in W_p$, $\lim_{n\rightarrow\infty, n\in E_p}b_s^n=\infty$,
\item[\rm(d)] for any $p,q\geq 1$, $n\in E_p$, $m\in E_q$ with $n\neq m$, $t\in W_q$ and $s\in I$ such that $\phi_n(s)=\phi_m(t)$,
$$
\frac{b_s^n}{b_t^m}\leq \frac{1}{M(p)M(q)}.
$$
\end{aitemize}
Moreover, one can replace ``there exists a sequence $(M(p))_{p\in\N}$ and a sequence 
$(E_{p})_{p\in\N}$'' by ``for any sequence $(M(p))_{p\in\N}$ there exists a sequence $(E_{p})_{p\in\N}$''.
\end{thm}

\proof
  We first observe that if properties (a) to (d) hold for some
  sequence $(M(p))_{p\in\N}$, then they are also satisfied for any
  sequence $(M(p))_{p\in\N}$, on considering, if necessary, a subsequence
  of $(E_p)_{p\in\N}$.

``$\Rightarrow$'':\hods Let $x\in c_0(I)$ be a frequently universal vector for $(T_n)_{n\in\N}$. Let $(\alpha_p)_{p\in\mathbb{N}}$ be a strictly increasing sequence of positive real numbers such that $\alpha_1=2$ and for all $p\geq2$, 
$\alpha_p> 4\alpha_{p-1}\rho^{2\varPsi(p)}$, where $\varPsi(p)=\max\{|g(t)|:\break t\in W_p\}$. Define
$$
F_p=\Bigl\{n\in\mathbb{N}: \|T_nx-\alpha_p\sum_{i\in W_p}e_i\|<\mfrac{1}{p}\Bigr\}.
$$
If $F_p=\{ n_k^p: k\in\N\}$, where $(n_k^p)_{k\in\N}$ is an increasing sequence of natural numbers, we define 
$E_p=\{ n^p_{(2[\varPsi(p)]+3)k}: k\in\N\}$ where $[\varPsi(p)]$ is the integer part of~$\varPsi(p)$. 

Clearly $\ldens(E_p)>0$ and the distance between two different elements of $E_p$ 
is greater than $2\varPsi(p)$. Moreover
\begin{equation}\label{claim}
\forall p\in \N\ \forall s\in W_p\ \forall n\in E_p: \quad 
\mfrac{\alpha_p}{2} \leq |b_s^nx_{\phi_n(s)}|<2\alpha_p.
\end{equation}
Indeed, 
$b_s^nx_{\phi_n(s)}$ is the $s$th coefficient of $T_nx$, so
$$
|b_s^nx_{\phi_n(s)}|\leq \Bigl\|T_nx-\alpha_p\sum_{i\in W_p}e_i\Bigr\|+ \alpha_p\Bigl\|\sum_{i\in W_p}e_i\Bigr\|
<\frac{1}{p}+\alpha_p<2\alpha_p,
$$
while 
\begin{align}\label{eqb_t}
|b_s^sx_{\phi_n(s)}|&\geq \alpha_p- |b_s^nx_{\phi_n(s)}-\alpha_p|\geq \alpha_p-\Bigl\|T_nx-\alpha_p\sum_{i\in W_p}e_i\Bigr\|\\[-6pt]
& \geq \alpha_p-\mfrac{1}{p}\geq \mfrac{\alpha_p}{2}.\notag
\end{align}
In particular,
\begin{equation}\label{claim1} 
\forall p\in \N\ \forall s\in W_p\ \forall n\in E_p: \quad x_{\phi_n(s)}\neq 0.
\end{equation}

In order to prove (b), 
fix $p\neq q$, with $p<q$, $n\in E_p,m\in E_q$ and assume 
by contradiction, that there exist $s\in W_p$ and $t\in W_q$ such that $\phi_n(s)=\phi_m(t).$ 
Then, by \eqref{claim},
$$
\frac{1}{\rho^{2\varPsi(q)}}\leq\frac{1}{\rho^{|n-m|}}\leq
\frac{|b_s^nx_{\phi_n(s)}|}{|b_t^mx_{\phi_m(t)}|}\leq 2\alpha_p\frac{2}{\alpha_q} \leq 4\frac{\alpha_{q-1}}{\alpha_q},$$
contradicting the choice of $(\alpha_p)_p$.

Now let $p\geq 1$ and $s\in W_p$. Let $M>0$. Given
$\varepsilon=\mfrac{\alpha_p}{(2M)}$, since $x\in c_0(I)$, there exists
$J\subseteq I$ finite such that $|x_i|<\varepsilon$ for all $i\in
I\setminus J$. Since ($\phi_n$) is a run-away sequence, there exists
$n_0\in\N$ such that for all $n\in\N$ with $n>n_0$ and  all $s\in W_p$ we have
$\phi_n(s)\notin J$, and so $|x_{\phi_n(s)}|<\varepsilon$. Hence,
for all $n\in E_p$ with $n\geq n_0$, by \eqref{claim} and \eqref{claim1},
$$
b_s^n\geq\frac{\alpha_p}{2|x_{\phi_n(s)}|}\geq\frac{\alpha_p}{2\varepsilon}=M.
$$
So, we have proved (c).

Finally, let $p,q\geq 1$, $n\in E_p$, $m\in E_q$ with $n\not=m$, $t\in
W_q$ and $s\in I$ be such that $\phi_n(s)=\phi_m(t)$. Then
$p\not=q$, as otherwise, since $n\not=m$, by the definition of~$E_p$ we have
$|n-m|>2\varPsi(p)$; on the other hand,  (iii) yields $|n-m|\leq
2\varPsi(p)$,  a contradiction.  Therefore, since
$p\not=q$, we can apply (b) to get $s\notin W_p$ and so the
$s$-coefficient of $T_nx-\alpha_p\sum_{i\in W_p}e_i$ is
$b_s^nx_{\phi_n(s)}$. Hence, by~\eqref{claim},
$$
\frac{b_s^n}{b_t^m}=\frac{|b_s^nx_{\phi_n(s)}|}{|b_t^mx_{\phi_m(t)}|}
\leq \frac{1}{p}\,\frac{2}{\alpha_q}\leq\frac{1}{p} \,\frac{1}{q}.
$$
Hence (d) holds with $M(p)=p$.

``$\Leftarrow$'':\hods
We first observe that if (b) holds, then
\begin{equation}\label{conditionb}
\forall p,q\in\N, \, p\not=q:\quad  E_p\cap E_q=\emptyset.
\end{equation}
Indeed, assume $p<q$; if there exists $n\in E_p\cap E_q$, then for any
$s\in W_p\subseteq W_q$, one gets $\phi_n(s)\in \phi_n(W_p)\cap
\phi_n(W_q)$, contradicting (b).

As properties (a) to (d) hold true for any sequence $(M(p))_{p\in\N}$,
we may assume that $M(p)\geq \rho^{4p}$ for any $p\geq 1$.

We set
$$
E_p'=E_p\setminus\bigcup_{s\in W_p}\{n\in\N: b_s^n\leq \rho^{4p}\}.
$$
\loo-1 By (c), $E_p'$ is a cofinite subset of $E_p$, hence
$\ldens(E_p')>0$. If $E_p'=\{ n_k^p: k\in\N\}$,\vspace{1pt} where $(n_k^p)_k$ is
an increasing sequence of natural numbers, we consider the set $G_p=\{
n^p_{(2[\varPsi(p)]+3)k}: k\in\N\}$. It has positive lower density
and moreover the distance between two different elements of $G_p$ is
greater than $2\varPsi(p)$.

Let $(y^p)_{p\geq 0}$ be a dense sequence in $c_0(I)$ such that supp$
(y^p)\subseteq W_p$ and $\|y^p\|<\rho^p$.  We define $x\in
\mathbb{R}^I$ by setting
\begin{equation}
x_i= \begin{cases}
\frac{1}{b_s^n}y^p(s) &\text{if $i=\phi_n(s),\, n\in G_p,s\in W_p$,}\\
0 &\text{otherwise}.
\end{cases}
\end{equation}
This definition is correct, because if $i=\phi_n(s)=\phi_m(t)$
with $n\in G_p$, $s\in W_p$, $m\in G_q$ and $t\in W_q$, then, by (b),
$p=q$, and assumption (iii) yields $|n-m|\leq|g(s)-g(t)|\leq
2\varPsi(p)$; hence, by the definition of $G_p$, $n=m$ and so 
$s=t$, by the injectivity of~$\phi_n$.

We have $x\in c_0(I)$. Indeed, given $\varepsilon>0$, there
exists $p_0\in\N$ such that for $p\geq p_0$ and $n\in G_p$, $s\in W_p$,
$i=\phi_n(s)$,
$$
|x_i|\leq \frac{\rho^p}{\rho^{4p}}\leq \varepsilon.
$$ 
If $p\leq p_0$, then
$$
|x_i|\leq \frac{{\rho}^{p_0}}{b_s^n}\rightarrow 0\quad\ \mbox{as}\ n\rightarrow \infty.
$$
We finally show that $x$ is a frequently hypercyclic vector by proving
that for all $p\geq 1$ and $n\in G_p$, $\|T_nx-y^p\|<\varepsilon(p)$ with
$\varepsilon(p)\rightarrow 0$ as $p\rightarrow \infty$.  We have
$$
\|T_nx-y^p\|=\sup_{s\notin W_p}|b_s^nx_{\phi_n(s)}|.
$$
If $s\notin W_p$, then $b_s^nx_{\phi_n(s)}$ does not vanish if and only if 
\begin{equation}\label{nm} 
\exists q\geq 1\ \exists m\in G_q,\, t\in W_q : \quad  \phi_n(s)=\phi_m(t).
\end{equation}
If \eqref{nm} holds, then $n\not= m$, as otherwise, $p=q$ by
\eqref{conditionb} and $s=t$ by the injectivity of $\phi_n$, which
is impossible since $s\notin W_p$ and $t\in W_p=W_q$.

Hence, we can apply (d) to get 
$$
|b_s^nx_{\phi_n(s)}|=\biggl|\frac{b_s^n}{b_t^m}y^q(t)\biggr|\leq \frac{\rho^q}{M(p)M(q)}
\leq\frac{\rho^q}{\rho^p\rho^q}=\frac{1}{\rho^p}.\squ
$$ 

As a corollary, we obtain a characterization of frequent
hypercyclicity for weighted backward shift operators defined on
$c_0(I)$ where $I\subseteq\R$.

\begin{cor}\label{fhc0V}
Let $I$ be a countably infinite subset of $\R$ such that $I+\mathbb{Z}\subseteq I$ $($resp. $I+\N\subseteq I)$, $I=\bigcup_{p=1}^{\infty} W_p$, where $(W_p)_p$ is an increasing sequence of finite subsets.
Let $(w_i)_{i\in I}$ be a  family of positive real numbers such that
 \begin{equation}\label{BB} 0<\inf_{i\in I}w_i \leq \sup_{i\in I} w_i <\infty.\end{equation} The operator $T:c_0(I)\rightarrow c_0(I)$ defined by
$T(x_i)_{i\in I}=(w_{i}x_{i+1})_{i\in I}$ is frequently hypercyclic on $c_0(I)$ if and only if there exist a sequence $(M(p))_{p\in\N}$ of positive real numbers tending to $\infty$ and a sequence $(E_p)_{p\in\N}$ of subsets of $\N$ such that

\begin{aitemize}{\rm(a)}
\item[\rm(a)] for any $p\geq 1$, $\ldens(E_p)>0$,
\item[\rm(b)] for any $p,q\geq 1$, $p\neq q$, $(E_p+W_p)\cap (E_q+W_q)=\emptyset$,
\item[\rm(c)] for all $p\geq 1$ and  $s\in W_p$, $\lim_{n\rightarrow\infty,\,n\in E_p}w_{s}
\ldots w_{s+n-1}=\infty$,
\item[\rm(d)] for any $p,q\geq 1$, $n\in E_p$ and $m\in E_q$ with $m\neq n$, and $t\in W_q$ 
$($resp. $t\in W_q$ such that $t+(m-n)\in I)$,
$$
\frac{w_{m-n+t}\ldots w_{m+t-1}}{w_{t}\ldots w_{t+m-1} }\leq\frac{1}{M(p)M(q)}.
$$
\end{aitemize}
Moreover, one can replace ``there exist a sequence $(M(p))_{p\in\N}$ and a sequence $(E_{p})_{p\in\N}$'' 
by ``for any sequence $(M(p))_{p\in\N}$ there exists a sequence $(E_{p})_{p\in\N}$''.
\end{cor}

\begin{proof}
Observe that $T$ is frequently hypercyclic if and only the family of its powers $(T^n)_{n\in\N}$ is frequently universal. If we set $T_n=T^n$, we have that 
\[
T_n[(x_i)_{i\in I}]=(w_{i}w_{i+1}\ldots w_{i+n-1}x_{i+n})_{i\in I}.
\]
Therefore, if we set 
\[
b_s^n:=w_{s}w_{s+1}\ldots w_{s+n-1},\quad\ \phi(s):=s+1,\quad\ \phi_n:=\underbrace{\phi\circ \dots\circ\phi}_n
\]
and $g(s):=s$ for all $s\in I$ and $n\in\N$, then assumptions (i), (iii) and (iv) of Theorem \ref{pseudoshift} are trivially satisfied, while (ii) follows from \eqref{BB}. The characterization follows by Theorem~\ref{pseudoshift}.
\end{proof}

\begin{rem}
Observe that condition (d) is equivalent to saying that 
for any $p,q\geq 1$, $n\in E_p$ and $m\in E_q$ with $n\neq m$, and $t\in W_q$ (resp. $t\in W_q$ such that $t+(m-n)\in I$),
\begin{equation}
\begin{cases}
w_{t}\ldots w_{t+m-n-1}\geq M(p)M(q) &\text{if $m>n$,}\\
w_{t+(m-n)}\ldots w_{t-2}w_{t-1}\leq \frac{1}{M(p)M(q)} &\text{if $m<n$},
\end{cases}
\end{equation}
and we obtain the conditions of \cite[Theorem 12]{bayartruzsa}.
\end{rem}

\section{Frequently hypercyclic translation semi\-groups}

The purpose of this section is to obtain a characterization of
frequent hypercyclicity for translation semigroups on
$C_0^{\rho}(\mathbb{R})$ under the assumption
$\sup_{k\in\Z}\mfrac{\rho(k+1)}{\rho(k)}\allowbreak <\infty$, and on
$L_p^\rho(\R)$.

In the following we set, for any $r,s\in \Z$, $\llbracket r,s\rrbracket=[r,s]\cap\Z$.

To treat the case of continuous functions, we recall some known
results about the construction of a Schauder basis in
$C_0(\mathbb{R})$, referring for more details to~\cite{doldeckmann1982schauder}.

Let $\widetilde D$ be the set of dyadic numbers except $1$, that is,
$\widetilde{D}=\bigcup_{n=0}^\infty D_n$ where $D_0=\{0\}$ and, if
$n\geq 1$,
\[ 
D_n=\biggl\{ \frac{2k-1}{2^n} : k=1,\ldots,2^{n-1}\biggr\}.
\] 
For any $\tau\in D_n$, set $\tau^{-}=\tau-2^{-n}$ and $\tau^{+}=\tau+2^{-n}$.

Let $\varphi(x):=\max\{0,1-|x|\}$ for $x\in \R$, and define for every
$k\in\Z$, $\tau\in D_n$, and $x\in\R$,
\begin{equation}\label{phi}
\varphi_{k+\tau}(x):=\varphi(2^n(x-k-\tau)).
\end{equation}
Observe that $\varphi_{k+\tau}(x)=\varphi_\tau(x-k)$ where
$\varphi_\tau$ is the Faber--Schauder dyadic function with peak at
$\tau$.

Set $I=\mathbb{Z}+\widetilde{D}$ and consider the partition
$I=\bigcup_{n\geq 0}V_n$ where $V_0=\{0\},$ and
\begin{equation}\label{basis}
V_n=\{-n+h+D_h: h=0,1,\ldots,n\}\cup\{h+D_{n-h}: h=1,\ldots,n\}.
\end{equation}
We define an order on $I$ assuming that the elements of $V_k$ are
earlier than those in $V_{n}$ if $0\leq k<n,$ and within each
$V_n$ we keep the usual order.

The system $(\varphi_i)_{i\in I}$ is a Schauder basis in $C_0(\mathbb{R})$. More precisely, if $f\in C_0(\mathbb{R})$, then $f=\sum_{k+\tau\in\Z+\widetilde{D}}a_{k+\tau}\varphi_{k+\tau}$ where
$$
a_{k+\tau}= \begin{cases}
f(k), &\text{$k\in\mathbb{Z},\, \tau=0$,}\\
f(k+\tau)-\frac{1}{2}(f(k+\tau^{-})+f(k+\tau^{+})),  &\text{$k\in \mathbb{Z},\,\tau\neq 0$.}
\end{cases}
$$

By the construction of the functions $\varphi_{k+\tau}$, it follows that  for any $n\in\Z_+$, for every family $(a_{k+\tau})_{k+\tau\in \Z+ D_n}$ of real numbers, and for every $x\in\R$:
\begin{equation}\label{atau} \left| \sum_{k\in\Z, \tau\in D_n} a_{k+\tau} \varphi_{k+\tau}(x)\right| \leq 2\sup_{k\in\Z, \tau\in D_n} |a_{k+\tau}|. \end{equation}

 \noindent If we set  for every $n\in\Z_+$ 
 \[ \widetilde D_n:=\bigcup_{h=0}^n D_h, \qquad W_n=\llbracket -n,n\rrbracket + \widetilde D_n,\]
then clearly $(W_n)_{n\geq 0}$ is an increasing  sequence of finite subsets such that  $\Z+\widetilde D=\bigcup_{n\geq 0} W_n$. Thus 
\begin{align*}&c_0(\Z+\widetilde{D})=\{(a_{k+\tau})_{k,\tau}\in \R^{\Z+\widetilde{D}}\mid\forall \varepsilon >0\ \exists n\in\N:  k+\tau\notin W_n \Rightarrow |a_{k+\tau}|<\varepsilon\}\\
&=\{(a_{k+\tau})_{k,\tau}\in \R^{\Z+\widetilde{D}}\mid\forall \varepsilon >0\ \exists n\in\N: ( |k|>n\ \mbox{or}\ \tau\notin \widetilde{D}_n) \Rightarrow |a_{k+\tau}|<\varepsilon\}.\end{align*}

For any $x\in\R$, let $[x]$ denote the integer part of $x$.

\begin{lem}\label{conjugation}
Let $\rho$ be an admissible weight function on $\R$ such that $\rho(x)=\rho([x])$ for any $x\in\R$ and 
let $T_1:C_0^{\rho}(\mathbb{R})\rightarrow C_0^{\rho}(\mathbb{R})$ be the translation operator 
defined as $T_1f(x)=f(x+1)$. Then $T_1$ is quasiconjugate to the weighted backward shift 
operator $B_w:c_0(\mathbb{Z}+\widetilde{D})\rightarrow c_0(\mathbb{Z}+\widetilde{D})$ defined by 
\[ 
B_w[(x_{k+\tau})_{k+\tau\in \Z+\widetilde D}]= (w_{k+\tau} x_{k+\tau+1})_{k+\tau\in \Z+\widetilde D},
\]
where 
\[
w_{k+\tau}:= \frac{\rho(k+\tau)}{\rho(k+1+\tau)}=\frac{\rho(k)}{\rho(k+1)},\quad\ 
k+\tau\in\Z+\widetilde {D}.
\]
Moreover, $B_w$ is  quasiconjugate to $T_1$.
\end{lem}

\begin{proof}
To prove that $T_1$ is quasiconjugate to $B_w$,  we exhibit a continuous linear operator $Q:C_0^\rho(\R) \rightarrow c_0(\mathbb{Z}+\widetilde D)$ with dense range
 such that the following diagram commutes: 
\[ 
\begin{array}{c@{\hskip3pt}c@{\hskip3pt}c}
C_0^\rho(\mathbb{R}) & \stackrel{T_1}{\longrightarrow}&C_0^\rho(\mathbb{R}) \\
{\downarrow}{\hskip1pt\scriptstyle Q}& &{\downarrow}{\hskip1pt\scriptstyle Q}\\
c_0(\mathbb{Z}+\widetilde{D}) &\stackrel{B_w}{\longrightarrow}&c_0(\mathbb{Z}+\widetilde{D}) \\
\end{array} 
\]

Given $f\in C_0^\rho(\mathbb{R})$, we define $Q(f)=(a_{k+\tau})_{k+\tau\in\mathbb{Z}+\widetilde{D}}$
where 
$$
a_{k+\tau}= \rho(k)\cdot \begin{cases} f(k),&\text{$k\in\mathbb{Z},\, \tau=0$,}\\
f(k+\tau)-\frac{1}{2}(f(k+\tau^{-})+f(k+\tau^{+})), &\text{$k\in \mathbb{Z},\,\tau\neq 0$.}
\end{cases}
$$
It holds that $Q(f)\in c_0(\Z+\widetilde D)$. Indeed, since $f\in C_0^{\rho}(\R)$, for every
$\varepsilon>0$,
\begin{itemize}
	\item[(1)]there exists $N_1\in\N$ such that if $|x|>N_1$, then
	$|f(x)\rho(x)|<\frac{\varepsilon}{2}$ 
	\item[(2)] there exists $\delta>0$  such that, for all $x,y \in [-N_1-1, N_1+2]$,  if  $|x-y|<\delta$ then $|f(x)-f(y)|<\frac{\varepsilon}{\rho_{N_1}}$, where $\rho_{N_1}=\max_{k\in \llbracket-N_1-1,N_1+2\rrbracket}{\rho(k)}$.
\end{itemize}

Choose  $N_2\in\N$ such that $\frac{1}{2^{N_2}}<\delta$ and set  $N=\max\{N_1+1,N_2\}$.  Let $k+\tau \in(\Z+\widetilde D) \setminus W_N$. 

If $|k|>N_1+1$, by (1)  it is clear that for every $s\in [0,1\mathclose{[}$,
\[ 
|f(k+s)\rho(k)|=|f(k+s)\rho(k+s)|<\frac \varepsilon 2,
\]
since $|k+s|>N_1$. Hence, by the continuity of $f$, for every $s\in [0,1]$,
\begin{equation}\label{s}
|f(k+s)\rho(k)|\leq \frac \varepsilon 2.
\end{equation}
By replacing $s$ by $\tau$, $\tau^+$ and $\tau^-$ in \eqref{s}, we
immediately get $|a_{k+\tau}|\leq \varepsilon$.

If $|k|\leq N_1+1$, then necessarily $\tau\in D_n$, with $n>N$, and,  in particular, $\tau\neq 0$. It holds that
\[ |\tau-\tau^{-}|<\frac{1}{2^N}<\delta, \   |\tau-\tau^{+}|<\frac{1}{2^N}<\delta, \]
and $ k+\tau, k+\tau^-, k+\tau^+\in [-N_1-1, N_1+2]$, 
 thus, by  (2) we get that 
 \[ |a_{k+\tau}|=|\rho(k)(f(k+\tau)-\tfrac{1}{2}(f(k+\tau^{-})+f(k+\tau^{+})))|\leq \varepsilon.\]

Then $Q: C_0^\rho(\R) \rightarrow c_0(\mathbb{Z}+\widetilde D)$ is
well-defined and clearly linear. Moreover, for every $f\in
C_0^\rho(\R)$, $k\in\Z$, we have 
\[ 
|f(k+s) \rho(k)| =|f(k+s)\rho(k+s)|\leq \|f\|^\rho_\infty\quad\ \mbox{for all}\ s\in [0,1\mathclose{[},
\]
hence, by the continuity of $f$,
\begin{equation}\label{s1}
|f(k+s) \rho(k)|\leq \|f\|^\rho_\infty\quad\ \mbox{for all}\ s\in [0,1].
\end{equation}
By replacing $s$ by $\tau$, $\tau^+$ and $\tau^-$ in \eqref{s1}, with
$\tau\in\widetilde D$, we immediately get  $\|Q(f)\|\leq
2\|f\|^\rho_\infty$, so $Q$ is continuous.

To prove that $Q$ has dense range, it is enough to show that for every
$h\in\Z$ and $ \sigma\in\widetilde D$, 
\[
(x^{h+\sigma}_{k+\tau})_{k\in\Z,\, \tau\in \widetilde D}\in Q(C_0^\rho(\R))\quad
\text{where}\quad x^{h+\sigma}_{k+\tau} =\delta_{(h,\sigma)}(k,\tau),
\]
and, indeed, by the definition of $\varphi_{h+\sigma}$, we have 
$$
Q(\rho(h)^{-1} \varphi_{h+\sigma})= (x^{h+\sigma}_{k+\tau})_{k\in\Z,\, \tau\in \widetilde D}.
$$
Finally, by observing that, for every $f\in C_0^\rho(\R)$, $Q\circ
T_1(f)=(b_{k+\tau})_{k+\tau\in\mathbb{Z}+\widetilde{D}}$, where
\begin{equation*}
b_{k+\tau}= \left\{
\begin{array}{@{}l}
f(k+1)\rho(k)=a_{k+1}\dfrac{\rho(k)}{\rho(k+1)} ,\hfill \mbox{if}\ \tau=0,\\
\begin{array}{@{}l@{}}( f(k+1+\tau)-\frac{1}{2}(f(k+1+\tau^{-})+f(k+1+\tau^{+}))\rho(k)\\[4pt]
\quad\ =a_{k+\tau+1}\dfrac{\rho(k)}{\rho(k+1)} \hfill \mbox{if}\ \tau\not=0,\end{array}
\end{array}
\right.
\end{equation*}
for $k\in\Z$ and $\tau\in\widetilde{D},$ it is immediate  that 
$Q\circ T_1(f) =B_w\circ Q(f)$.\\

Now, let us prove that $B_w$ is quasiconjugate to $T_1$, namely that  there exists a continuous linear operator $P:c_0(\mathbb{Z}+\widetilde D) \rightarrow C_0^\rho(\R)$ with dense range
such that the following diagram commutes: 
\[ 
\begin{array}{c@{\hskip3pt}c@{\hskip3pt}c}
c_0(\mathbb{Z}+\widetilde{D}) & \stackrel{B_w}{\longrightarrow}&c_0(\mathbb{Z}+\widetilde{D}) \\
{\downarrow}{\hskip1pt\scriptstyle P}& &{\downarrow}{\hskip1pt\scriptstyle P}\\
C_0^\rho(\R) &\stackrel{T_1}{\longrightarrow}&C_0^\rho(\R) \\
\end{array} 
\]

Given $({a_{k+\tau}})_{k\in\mathbb{Z}, \tau\in\widetilde{D}}\in c_0(\mathbb{Z}+\widetilde{D})$, we define \begin{equation}\label{series}P\left((a_{k+\tau})_{k\in\mathbb{Z}, \tau\in\widetilde{D}}\right)=\sum_{n=0}^\infty\frac{1}{2^n}\left(\sum_{k\in\Z, \tau\in D_n}\frac{a_{k+\tau}}{\rho(k)}\varphi_{k+\tau}\right).\end{equation}
 Observe that, for every $N\in\N$,  $x\in[-N,N]$, taking \eqref{atau} into account,
\begin{align*} &\left|\sum_{k\in\Z, \tau\in D_n}\frac{a_{k+\tau}}{\rho(k)}\varphi_{k+\tau}(x)\right|
\leq
\left|\sum_{|k|\leq N, \tau\in D_n}\frac{a_{k+\tau}}{\rho(k)}\varphi_{k+\tau}(x)\right|\leq\\
&\leq \sum_{|k|\leq N, \tau\in D_n}\frac{|a_{k+\tau}|}{\rho(k)}\varphi_{k+\tau}(x)\leq 2\frac{\|(a_{k+\tau})_{k,\tau}\|}{\min_{|k|\leq N}\{\rho(k)\}},
\end{align*}
thus by the Weierstrass M-test, the series on the right-hand  side of \eqref{series} is uniformly convergent on $[-N,N]$. 
It follows that $P\big((a_{k+\tau})_{k\in\mathbb{Z}, \tau\in\widetilde{D}}\big)$ is a continuous function on $\R$.

 Moreover, for every $x\in\R$ there exists $\overline{k}\in\Z$  such that $\overline{k}\leq x<\overline{k}+1$. Hence, by \eqref{atau} and by \eqref{admissibility} applied with $l=1$
\begin{align}\label{overlinek} \nonumber |P\left((a_{k+\tau})_{k,\tau}\right)(x)\rho(x)|=\left|\rho(x)\sum_{n=0}^\infty\frac{1}{2^n}\left( \sum_{k\in\Z, \tau\in D_n}\frac{a_{k+\tau}}{\rho(k)}\varphi_{k+\tau}(x)\right)\right|= \\
=\left| \rho(\overline k) \left( \sum_{k\in\Z}\frac{a_{k}}{\rho(k)}\varphi_{k}(x)\right) + 
\rho(\overline k) \sum_{n=1}^\infty\frac{1}{2^n}\left( \sum_{k\in\Z, \tau\in D_n}\frac{a_{k+\tau}}{\rho(k)}\varphi_{k+\tau}(x)\right)\right|\leq\nonumber\\
\leq \left| a_{\overline k} \varphi_{\overline k}(x)+ \frac{\rho(\overline k)}{\rho(\overline k+1)} a_{\overline k +1}\varphi_{\overline k +1}(x) + \rho(\overline k)\sum_{n=1}^\infty\frac{1}{2^n}\left( \sum_{ \tau\in D_n}\frac{a_{\overline{k}+\tau}}{\rho(\overline{k})}\varphi_{\overline{k}+\tau}(x)\right)\right|\leq \nonumber\\
\leq\frac{\rho(\overline k)}{\rho(\overline k+1)} \left|a_{\overline k +1}\right| + \left|\sum_{n=0}^\infty\frac{1}{2^n}\left(\sum_{ \tau\in D_n}a_{\overline{k}+\tau}\varphi_{\overline{k}+\tau}(x)\right)\right|\leq \nonumber\\
\leq B\left| a_{\overline k +1}\right| + 2\sum_{n=0}^\infty\frac{1}{2^n}\sup_{\tau\in D_n}|a_{\overline{k}+\tau}|. \end{align}

Thus, if $\varepsilon>0$ and  $N\in\N$  is such that  $|a_{k+\tau}|<\varepsilon$ if $k+\tau\notin W_N$, then 
  for any $x\in\R$ such that $|x|>N+1$,  it holds that   $|\overline k|>N$ and  $|\overline k+1|>N$, thus
  $\overline k+\tau \notin W_N$ for every $\tau\in\widetilde D$ and $\overline k +1\notin W_N$. Therefore $|a_{\overline k+\tau}|<\varepsilon$ and $|a_{\overline k +1}| <\varepsilon$.
      Hence 
\begin{align*}& |P\left((a_{k+\tau})_{k,\tau}\right)(x)\rho(x)|<(B+4)\varepsilon.\end{align*}
On the other hand, \eqref{overlinek} implies that 
$$\|P\left((a_{k+\tau})_{k,\tau}\right) \|_{\infty}^\rho\leq (B+4)\|(a_{k+\tau})_{k,\tau}\|,$$
 hence $P:c_0(\mathbb{Z}+\widetilde D) \rightarrow C_0^\rho(\R)$ is well defined,  clearly linear and continuous.

  Observe that clearly
  $$\{\varphi_{k+\tau}, k\in\Z, \tau\in\widetilde{D}\}\subseteq P(c_0(\Z+\widetilde{D})).$$
 If $f\in C_0^\rho(\R)$, given $\varepsilon>0$, there exists $g\in C(\R)$ with compact support such that  $\|f-g\|_{\infty}^\rho<\frac{\varepsilon}{2}$. Let us assume that ${\rm supp} (g)\subseteq [-M-1,M+1]$, $M\in\N$ and consider $\varphi\in {\rm span}\{\varphi_{k+\tau}, k\in\Z, \tau\in\widetilde{D}\}\subseteq P(c_0(\Z+\widetilde{D}))$, $ {\rm supp}(\varphi)\subseteq [-M-1,M+1]$,  such that $\|g-\varphi\|_\infty<\frac{\varepsilon}{2\max_{x\in[-M-1,M+1]}\rho(x)}$. Then
  \[\|f-\varphi\|_{\infty}^\rho<\|f-g\|_{\infty}^\rho+\|g-\varphi\|_{\infty}^\rho<\frac{\varepsilon}{2}+\sup_{x\in[-M-1,M+1]}|g(x)-\varphi(x)|\rho(x)<\varepsilon\] and therefore $P$ has dense range.\\

 Finally  we observe that 
 \[P\circ B_w ((a_{k+\tau})_{k,\tau})=\sum_{n=0}^\infty\frac{1}{2^n}\left(\sum_{k\in\Z, \tau\in D_n}\frac{a_{k+\tau+1}}{\rho(k+1)}\varphi_{k+\tau}\right).\]
  On the other hand,  for every $x\in\R$,  
  \begin{align*}& T_1\circ P((a_{k+\tau})_{k,\tau})(x)=\sum_{n=0}^\infty\frac{1}{2^n}\left(\sum_{k\in\Z, \tau\in D_n}\frac{a_{k+\tau}}{\rho(k)}\varphi_{k+\tau}(x+1)\right)=\\
  &\sum_{n=0}^\infty\frac{1}{2^n}\left(\sum_{k\in\Z, \tau\in D_n}\frac{a_{k+\tau}}{\rho(k)}\varphi_{k+\tau-1}(x)\right)=
  \sum_{n=0}^\infty\frac{1}{2^n}\left(\sum_{h\in\Z, \tau\in D_n}\frac{a_{h+\tau+1}}{\rho(h+1)}\varphi_{h+\tau}(x)\right).
  \end{align*} Then  $P\circ B_w=T_1\circ P$.

\end{proof}

\begin{thm}\label{c0fh}
Let $\mathcal T$ be the translation semigroup on $C_0^\rho(\mathbb{R})$, where $\rho$ is an admissible weight function and $\sup_{k\in\mathbb{Z}}\mfrac{\rho(k+1)}{\rho(k)}<\infty$. Then
$\mathcal T$ is frequently hypercyclic on $C_0^\rho(\mathbb{R})$ if and only if 
there exist a sequence $(M(p))_{p\in\N}$ of positive real numbers tending to $\infty$ and a sequence $(E_p)_{p\in\N}$ of subsets of $\N$ such that:

\begin{aitemize}{\rm(a)}
\item[\rm(a)] for any $p\geq 1$, $\ldens(E_p)>0$,
\item[\rm(b)] for any distinct $p,q\geq 1$,  $(E_p+\llbracket -p, p\rrbracket)\cap (E_q+\llbracket -q, q\rrbracket)=\emptyset$,
\item[\rm(c)] for any $p\geq 1$, $\lim_{n\rightarrow\infty,\,n\in E_p}\rho(n)=0$,
\item[\rm(d)] for any $p,q\geq 1$ and any $n\in E_p$ and  $m\in E_q$ with $m\neq n$,\vspace*{-4pt}\end{aitemize}
\begin{equation}
\rho(m-n)\leq \frac{1}{M(p)M(q)}.
\end{equation}
Moreover, one can replace ``there exist a sequence $(M(p))_{p\in\N}$ and a sequence $(E_{p})_{p\in\N}$'' 
by ``for any sequence $(M(p))_{p\in\N}$ there exists a sequence $(E_{p})_{p\in\N}$''.
\end{thm}

\proof
\hskip-3pt Since $\rho$ is an admissible weight function $\sup_{k\in\Z} \mfrac{\rho(k)}{\rho(k+1)}\!<\!\infty,$
 there exists $M>1$ such that for every $k\in\Z$,
\[ 
M^{-1} \leq \frac{\rho(k+1)}{\rho(k)}\leq M.
\]
As a consequence we get for every $ n\in\N$ we get, if $k\in \llbracket0,p\rrbracket$,
\[ 
\rho(k+n)= \frac{\rho(k+n)}{\rho(k-1+n)}\cdots \frac{ \rho(n+1)}{\rho(n)}\cdot \rho(n) \leq M^k \rho(n)
\]
and if $k\in \llbracket-p,0\rrbracket$,
\[ 
\rho(k+n)= \frac{\rho(k+n)}{\rho(k+1+n)}\cdots \frac{ \rho(n-1)}{\rho(n)} \cdot \rho(n) \leq M^{|k|} \rho(n),
\]
thus for every $k\in\llbracket-p,p\rrbracket$, 
\begin{equation}\label{k+n} 
\rho(k+n) \leq M^{|k|} \rho(n).
\end{equation}

On the other hand, by \eqref{admissibility}, there exist constants $0<A<B$ such that for every $x\in [k,k+1]$,
$$
A\rho(k)\leq\rho(x)\leq B\rho(k+1)\leq BM \rho(k).
$$
If we now define $\tilde{\rho}(x)=\rho([x])$ for $x\in\R$, then there exist constants $M_1,M_2>0$ such that
$$
M_1\|f\|_{\infty}^{\tilde{\rho}}\leq\|f\|_{\infty}^{\rho}\leq M_2\|f\|_{\infty}^{\tilde{\rho}},
$$
so $\|\cdot \|_\infty^{\widetilde\rho}$ is an equivalent norm to
$\|\cdot\|_\infty^\rho$.  Therefore, without loss of generality we can
assume in the following that $\rho(x)=\rho([x])$ for every $x\in\R$.

By the results in
\cite{conejero_muller_peris,mangino_peris2011frequently}, it is known
that $\mathcal T$ is a frequently hypercyclic semigroup if and only if
$T_1$ is a frequently hypercyclic operator. Since $T_1$ is quasiconjugate
to the operator $B_w$ defined as in Lemma~\ref{conjugation} and $B_w$ is quasiconjugate to $T_1$, by 
\cite[Proposition 9.4]{grosse-erdmann_peris2011linear} we find that
$\mathcal T$ is frequently hypercyclic if and only if $B_w$ is
frequently hypercyclic, hence, by Corollary~\ref{fhc0V}, if and only
if there exist a sequence 
$(M(p))_{p\in\N}$ of positive real numbers tending to $\infty$ and
a sequence $(E_p)_{p\in\N}$ of subsets of $\N$ (or equivalently, for any $(M(p))_{p\in\N}$ there exists
$(E_p)_{p\in\N}$) such that

\begin{aitemize}{\rm(a1)}
\item[(a1)] for any $p\geq 1$, $\ldens(E_p)>0$,
\item[(b1)] for any distinct $p,q\geq 1$,  $(E_p+W_p)\cap (E_q+W_q)=\emptyset$,
\item[(c1)] for every $p\geq 1$ and every $s\in W_p$, 
$$
\lim_{n\rightarrow\infty,\,n\in E_p}w_{s}\ldots w_{s+n-1}=\lim_{n\to\infty,\, n\in E_p} 
\frac{\rho(s)}{\rho(s+n)}=\infty,
$$
\item[(d1)] for any $p,q\geq 1$, $n\in E_p$ and $m\in E_q$ with $m\neq n$, and $t\in W_q$, 
$$
\frac{w_{m-n+t}\ldots w_{m+t-1}}{w_{t}\ldots w_{t+m-1} }
=\frac{\rho(m-n+t)}{\rho(t)}\leq\frac{1}{M(p)M(q)},
$$
\end{aitemize}
where $W_p=\llbracket -p,p\rrbracket + \widetilde D_p$. 

Clearly (a1)--(d1) imply  (a)--(d)  simply by
observing that $\llbracket-p,p\rrbracket\subseteq W_p$, and in
particular $0\in W_p$, for every $p\in\N$.

\loo-1 Conversely, assume that (a)--(d) hold. Passing to a subsequence of
$(E_p)_{p\in\N}$ if necessary, we can choose $(M(p))_{p\in\N}$ such
that
$$
\lim_{p\to\infty}\frac{ M(p) \cdot \min\{\rho(k): k\in \llbracket-p,p\rrbracket\}}{M^p}=\infty.
$$
If 
\[
( E_p + W_p)\cap( E_q+W_q)\not=\emptyset,
\]
then there exist $n\in E_p$, $s\in \llbracket-p,p\rrbracket$,
$\sigma\in \widetilde D$, $m\in E_q$, $t\in \llbracket-q,
q\rrbracket$ and $\tau \in \widetilde D$, such that
\[ 
n+s+\sigma=m+t+\tau;
\]
taking integer parts yields $n+s=m+t$, hence
$(E_p+\llbracket-p,p\rrbracket)\cap
(E_q+\llbracket-q,q\rrbracket)\break \not=\emptyset$.  Thus $p=q$. So (b1) is
satisfied.

By \eqref{k+n} and (c), we get  $\lim_{n\to\infty, n\in E_p}
\rho(n+k)=0$ for every $k\in\llbracket-p,p\rrbracket$, and (c1)
follows by observing that for every $s\in W_p$ there exists $k\in
\llbracket-p,p\rrbracket$ such that $\rho(n+s)=\rho(n+k)$.

Let $p,q\geq 1$ and $n\in E_p$, $m\in E_q$, $m\not=n$. 
By \eqref{k+n} and (d), 
for every $k\in \llbracket-q,q\rrbracket$,
\begin{align*} 
\frac{\rho(m-n+k)}{\rho(k)}&\leq \frac{\rho(m-n)}{\rho(k)} M^{|k|} \leq \frac{M^q }{\min\{\rho(k): k\in \llbracket-q,q\rrbracket\}M(p)M(q)}\\
&\leq\frac{1}{K(p)K(q)} 
\end{align*}
where 
$$
K(q)=\min\biggl\{ \frac{ M(q)\cdot \min\{\rho(k): k\in \llbracket-q,q\rrbracket\}}{M^q}, M(q)\biggr\}.
$$
We get (d1) by observing that for every $t\in W_q$ there exists
$k\in \llbracket-q,q\rrbracket$ such that
\[
\frac{\rho(m-n+t)}{\rho(t)}=\frac{\rho(m-n+k)}{\rho(k)}.\squ
\]

\begin{rem}\label{precise}
As an immediate consequence,  if $\rho$ is an admissible weight function on $\R$ such that 
$\sup_{k\in\mathbb{Z}}\mfrac{\rho(k+1)}{\rho(k)}<\infty$ and we set $w_k=\break \mfrac{\rho(k)}{\rho(k+1)}$ for $k\in\mathbb{Z}$, then, by the characterization given in \cite[Theorem 9]{bayartruzsa} and by Theorem \ref{c0fh}, 
$B_w$ is frequently hypercyclic on $c_0(\mathbb{Z})$ if and only if the translation semigroup is frequently hypercyclic on $C_0^\rho(\R)$.
\end{rem}

\begin{prop}
If the translation semigroup $\mathcal T$ is mixing $($equivalently chaotic$)$
 on $C_0^\rho(\mathbb{R})$, then it is frequently hypercyclic.
\end{prop}

\begin{proof}
As  is proved in \cite{bermudez_et_alt,matsuitakeo}, chaos and mixing are equivalent properties for the translation $C_0$-semigroup on $C_0^{\rho}(\mathbb{R})$, and this happens if and only if $\lim_{x\rightarrow\pm \infty}\rho(x)=0$. 

As  already observed, it is enough to prove that $T_1$
satisfies the Frequent Hypercyclicity Criterion for operators (see
\cite{bonilla_grosse-erdmann2007frequently}).  Let $X_0=\mathop{\rm span} \{
\varphi_{k+\tau}: k\in\Z,\, \tau\in \widetilde D\}$, where
$\varphi_{k+\tau}$ is defined in \eqref{phi}. Every continuous
function on $\R$ with dense support can be approximated in the uniform
norm with elements of~$X_0$, and therefore $X_0$ is dense in
$C_0^\rho(\R)$.  Moreover, if we define $S:X_0\rightarrow X_0$ by
$Sf(x)=f(x-1)$, it is clear that $T_1Sf=f$. 

Let us prove that
$\sum_{n=1}^\infty T_1^nf$ and $\sum_{n=1}^\infty S^n f$ are
unconditionally convergent for all $f\in X_0$.  It is enough to
consider $f=\varphi_{k+\tau}\in X_0$; then $\mathop{\rm supp}(f)\subseteq [a,b]$
with $b-a\leq 2$, and for every $n\in\N$,
\[ 
{\rm supp}(T_1^n(f))\subseteq [a-n, b-n];
\]
thus for every $n,m\in\N$, if supp$(T_1^nf)
\cap \mathop{\rm supp}(T_1^mf)\not=\emptyset$, then it is immediate  that
$|n-m|\leq |b-a|\leq 2$, hence either $m=n$, or  $m=n\pm 1$, or
$m=n\pm 2$. This implies that if $J\subseteq\N$ is finite, then
\[ 
\Bigl\|\sum_{n\in J}T_1^nf\Bigr\|_\infty^\rho \leq 4\sup_{n\in J}\|T_1^nf\|_\infty^\rho.
\]

Let $\varepsilon >0$ and let $M>0$ be such that $\rho(x)<\varepsilon $
for every $|x|>M$. For every finite set $F\subseteq \N\cap
\mathopen{]}M+b,\infty\mathclose{[}$ and every $x\in [a,b]$ and $n\in F$, we have
 $|x-n|=n-x>M$, so
\[ 
\Bigl\|\sum_{n\in F}T_1^nf\Bigr\|_\infty^\rho \leq 4\sup_{x\in\R, \,n\in F} |f(x+n)\rho(x)|
=2\sup_{x\in [a,b],\, n\in F} |f(x)\rho(x-n)|\leq 
2\varepsilon.
\]
The argument for $\sum_{n=1}^\infty S^nf$ is similar.
\end{proof}

\begin{rem}
The converse of the previous proposition does not hold. Indeed, let $(w_k)_{k\in \mathbb{Z}}$ 
be one of the sequences of weights constructed in \cite{bayartruzsa} such that $B_w$ is 
frequently hypercyclic on $c_0(\mathbb{Z})$ and $w_1\cdots w_k=1$ for infinitely many~$k$.
Define $\rho(k)=(w_1\cdots w_k)^{-1}$ if $k\geq 1$, $\rho(k)=w_k w_{k+1}\cdots w_0$ if $k\leq 0$, 
and $\rho(x)=\rho([x])$ for any $x\in\mathbb{R}$. Then
$\sup_{k\in\Z}\mfrac{\rho(k+1)}{\rho(k)}<\infty$, due to the fact that 
$\mfrac{1}{2}\leq w_k\leq 2$ for all $k\in\Z$, as shown in~\cite{bayartruzsa}.
By Remark \ref{precise}  the translation semigroup is frequently 
hypercyclic on $C_0^\rho(\R)$, while clearly it is not mixing, since $\rho(k)=1$ for infinitely many~$k$.
\end{rem}

Finally, set $J=\mathbb{\Z_+}+\widetilde{D}$ and consider the
partition $J=\bigcup_{n\geq 0}J_n$ where $J_0=\{0\}$ and
\begin{equation}\label{basis2} 
J_n=\{h+D_{n-h}: h=0,1,\ldots,n\}\quad\ \text{for}\ n\ge 1.
\end{equation}
Define an order on $J$ assuming that the elements $J_k$ are earlier
than those in $J_{n}$ if $0\leq k<n,$ and within each $J_n$ we
keep the usual order.

The system $(\psi_i)_{i\in J}$, where $\psi_0(x):=\max\{0,1-x\}$ and
$\psi_{k+\tau}= \varphi_{k+\tau}$ if $\tau\not=0$ or $k\not=0$, is
a Schauder basis on $C_0([0,\infty\mathclose{[}).$

Reasoning analogously to the case of translation semigroups in
$C_0^\rho(\R)$, we also get the following characterization of
frequently hypercyclic translation semigroups on
$C_0^\rho([0,\infty\mathclose{[})$:

\begin{thm}\label{c0fhinfty}
Let $\mathcal T$ be the translation semigroup on $C_0^\rho([0,\infty\mathclose{[})$, 
where $\rho$ is an admissible weight function and $\sup_{k\in\mathbb{N}}\mfrac{\rho(k+1)}{\rho(k)}<\infty$.
Then $\mathcal T$ is frequently hypercyclic on $C_0^\rho([0,\infty\mathclose{[})$ if and only if 
there exist a sequence $(M(p))_{p\in\N}$ of positive real numbers tending to $\infty$ 
and a sequence $(E_p)_{p\in\N}$ of subsets of $\N$ such that:

\begin{aitemize}{\rm(a)}
\item[\rm(a)] for any $p\geq 1$, $\ldens(E_p)>0$,
\item[\rm(b)] for any distinct $p,q\geq 1$, $(E_p+ \llbracket 0, p\rrbracket)\cap (E_q+ \llbracket 0,q\rrbracket )=\emptyset$,
\item[\rm(c)] for every $p\geq 1$, $\lim_{n\rightarrow\infty,\,n\in E_p}\rho(n)=0$,
\item[\rm(d)] for any $p,q\geq 1$ and any $n\in E_p$ and $m\in E_q$ with $m> n$,\vspace*{-6pt}\end{aitemize}
\begin{equation}
\rho(m-n)\leq \frac{1}{M(p)M(q)},
\end{equation}
Moreover, one can replace ``there exist a sequence $(M(p))_{p\in\N}$ and a sequence $(E_{p})_{p\in\N}$'' 
by ``for any sequence $(M(p))_{p\in\N}$ there exists a sequence $(E_{p})_{p\in\N}$''.
\end{thm}

\loo-1 The final part of the paper will be devoted to characterizing frequently
hypercyclic semigroups on $L_p^\rho(\R)$. Also in this case we will
first establish a relation between the discrete and the continuous
cases.  We recall that the relation between the discrete and the continuous
cases for Devaney chaos was studied in \cite{bayart_bermudez} and for
distributional chaos in \cite{barrachina_peris}.  The following lemma
follows immediately from the conjugacy of the backward shift $B$ on
$\ell_p^v=\{(x_k)_{k\in\Z}:\sum_{k\in\Z}|x_k|^pv_k<\infty\}$ and the
weighted backward shift $B_w$ on $\ell_p$ where $
w_k=(\mfrac{v_{k}}{v_{k+1}})^{\mfrac{1}{p}}$ for
$k\in\mathbb{Z}$, and from the characterization of frequently hypercyclic
weighted backward shifts on $\ell_p$ proved in~\cite[Theorem~3]{bayartruzsa}.

\begin{lem}\label{backwardfh}
Let $v=(v_k)_{k\in\mathbb{Z}}$ be a sequence of strictly positive weights such that ${(\mfrac{v_{k}}{v_{k+1}})}_k$ is bounded. Then the backward shift operator $B$ is frequently hypercyclic on $\ell_p^v$ if and only if\/ $\sum_{k\in\Z} v_k<\infty$.
\end{lem}

\begin{thm}\label{fhp}
\spaceskip.34em plus.14em minus.21em Let $\rho$ be an admissible weight function on $\R$.
If the translation semigroup $\mathcal T$ is frequently hypercyclic on $L_p^\rho(\R)$, then the backward shift operator $B$ is frequently hypercyclic on $\ell_p^v$, where $v_k=\rho(k)$ for all $k\in\mathbb{Z}$.
\end{thm}

\begin{proof}
Since $\rho$ is an admissible weight function, by \eqref{admissibility} there exist $A,B\geq 0$ 
such that $A\rho(k)\leq\rho(t)\leq B\rho(k+1)$ for all $t\in[k,k+1]$. If $(T_t)_{t\geq0}$ is frequently hypercyclic, then $T_1$ is frequently hypercyclic \cite{conejero_muller_peris}. Hence there exists $f\in L_p^{\rho}(\R)$ such that for all $g\in L_p^{\rho}(\R)$ and for all $\varepsilon>0$,\vspace{-2pt} 
\[
\ldens(\{n\in\mathbb{N}: \|T_1^nf-g\|<\varepsilon\})>0.\vspace{-2pt}
\]
Since $f\in L_p^{\rho}(\R)$ we have  $|f|{\rho}^{\mfrac{1}{p}}\in L_p([k,k+1])\subseteq L_1([k,k+1])$ for every $k\in\mathbb{Z}$. As $\rho$ is strictly positive and locally bounded below  by \eqref{admissibility}, we infer
 that $f\in L_1([k,k+1])$ for all $k\in\mathbb{Z}$. Therefore we can define 
$x_k=\int_{k}^{k+1}f(t)\,dt$ for all $k\in\mathbb{Z}$. We have\vspace{-4pt} 
\begin{eqnarray*}
\sum_{k\in\Z}|x_k|^p\rho(k)&=&\sum_{k\in\Z}\Big|\int_{k}^{k+1}f(t)\,dt\Big|^p\rho(k)
\leq \sum_{k\in\Z}\int_k^{k+1}|f(t)|^p\rho(k)\,dt\\[-2pt]
&\leq& \frac{1}{A}\sum_{k\in\Z}\int_k^{k+1}|f(t)|^p\rho(t)\,dt=\frac{1}{A}\|f\|^p_p<\infty.
\end{eqnarray*}
So $x=(x_k)_{k\in\Z}\in\ell^p_v$ with $v_k=\rho(k)$. 

Let
\[
y=(0,\ldots,y_{-N},\ldots,y_0,\ldots,y_M,0,\ldots,0)
\]
and let $\varepsilon>0$. Set $ g=\sum_{k=-N}^My_k\chi_{[k,k+1]}\in L_p^\rho(\R).$
We will show that
$$
\{n\in\mathbb{N}:\|T_1^nf-g\|<A^{\mfrac{1}{p}}\varepsilon\}\subseteq\{n\in\mathbb{N}:\|B^nx-y\|<\varepsilon\},
$$
and therefore
$$
\ldens(\{n\in\mathbb{N}: \|B^nx-y\|<\varepsilon\})>0
$$
because $f$ is a frequently hypercyclic vector. We have
\begin{align*}
\|B^nx-y\|^p&=\sum_{k\in\Z}|x_{n+k}-y_k|^p\rho(k)\\
&\leq \frac{1}{A}\sum_{k\in\Z}\int_{k}^{k+1}|f(t+n)-g(t)|^p\rho(t)\,dt\leq \varepsilon^p.
\end{align*}
By the density of finite sequences in $\ell_p^v$ we conclude that $B$ is frequently hypercyclic.
\end{proof}

Finally, we are able to characterize frequently hypercyclic translation semigroups on $L_p^\rho(\R)$.

\begin{thm}\label{lpfh}
Let $\rho$ be an admissible weight function on $\R$. The following assertions are equivalent:

\wc{\rm(2)}
\begin{itemize}
\item[\rm(1)]The translation semigroup $\mathcal{T}$ is frequently hypercyclic on $L_p^\rho(\R)$.
\item[\rm(2)]$\sum_{k\in\Z} \rho(k)<\infty$.\vspace{2pt}
\item[\rm(3)]$\int_{-\infty}^\infty\rho(t)\,dt<\infty$.
\item[\rm(4)] $\mathcal T$ is chaotic on $L_p^\rho(\R)$.
\item[\rm(5)] $\mathcal T$ satisfies the Frequent Hypercyclicity Criterion.
\end{itemize}
\end{thm}

\begin{proof}
Observe that $(\mfrac{\rho(k)}{\rho(k+1)})_{k\in\Z}$ is bounded by the admissibility of~$\rho$.
By Theorem \ref{fhp} and Lemma \ref{backwardfh}, we have (1)$\Rightarrow$(2). The equivalence of $(2)$ and $(3)$ follows from the properties of $\rho$, by comparing integrals and series.
The equivalence of $(4)$ and $(5)$ can be proved with the same argument as in \cite[Proposition 3.3]{mangino_peris2011frequently}. (5)$\Rightarrow$(1) is proved in \cite[Theorem 2.2]{mangino_peris2011frequently}, while (3)$\Rightarrow$(5) can be proved as in \cite[Proposition 3.4]{mangino_peris2011frequently}.
\end{proof}

With minor changes we also get a characterization of frequently
hypercyclic translation semigroups on $L_p^\rho([0,\infty\mathclose{[})$:

\begin{thm}
Let $\rho$ be an admissible weight function on $[0,\infty\mathclose{[}$. 
The following assertions are equivalent:

\wc{\rm(2)}
\begin{itemize}
\item[\rm(1)]The translation semigroup $\mathcal{T}$ is frequently hypercyclic on $L_p^\rho([0,\infty\mathclose{[})$.
\item[\rm(2)]$\sum_{k\in\N} \rho(k)<\infty$.\vspace{2pt}
\item[\rm(3)]$\int_0^\infty\rho(t)\,dt<\infty$.\vspace{2pt}
\item[\rm(4)] $\mathcal T$ is chaotic on $L_p^\rho([0,\infty\mathclose{[})$.
\item[\rm(5)] $\mathcal T$ satisfies the Frequent Hypercyclicity Criterion.
\end{itemize}
\end{thm}

{\bf Acknowledgements.}\hods
The authors thank the anonymous referee for the critical reading of the manuscript and for his constructive comments.

This work is supported by MEC and FEDERER, proyect MTM 2013-47093-P.
The second author also acknowledges the support of a grant by the FPU
Program of Ministerio de Educaci\'on and a travel grant from the
Foundation Ferran Sunyer i Balaguer, and she thanks the Dipartimento di
Matematica e Fisica ``E. De Giorgi'' (Lecce, Italy) for the hospitality
during her stay.

\end{document}